\documentclass[final,1p,times]{elsarticle}

\usepackage{amsthm,amsmath,amssymb,mathrsfs,amsfonts,graphicx,graphics,latexsym,exscale,cmmib57,dsfont,bbm,amscd,ulem}
\usepackage{color,soul}
\usepackage[colorlinks=true]{hyperref}

\newcommand{\bA}{\boldsymbol{A}}
\newcommand{\bB}{\boldsymbol{B}}
\newcommand{\bK}{\boldsymbol{K}}
\newcommand{\brho}{\boldsymbol{\rho}}
\newcommand{\bcdot}{\boldsymbol{\cdot}}

\newcommand{\lam}[1]{\lambda_}

\newcommand{\bea}{\begin{eqnarray}}
\newcommand{\eea}{\end{eqnarray}}
\newcommand{\bean}{\begin{eqnarray*}}
\newcommand{\eean}{\end{eqnarray*}}
\newtheorem{prob}{Problem}

\newtheorem*{thm-A}{Theorem A}
\newtheorem*{thm-B}{Theorem B}
\newtheorem*{thm-C}{Theorem C}
\newtheorem*{thm-D}{Theorem D}
\newtheorem*{thm-E}{Theorem E}
\newtheorem*{thm-F}{Theorem F}
\newtheorem*{thm-G}{Theorem G}
\newtheorem*{thm-H}{Theorem H}

\newtheorem{theorem}{Theorem}
\newtheorem{cor}[theorem]{Corollary}
\newtheorem{prop}[theorem]{Proposition}

\theoremstyle{definition}

\newtheorem{example}[theorem]{Example}

\numberwithin{equation}{section}

\journal{Linear Algebra and its Applications}%

\begin{document}

\begin{frontmatter}

\title{Some criteria for spectral finiteness of a finite subset of the real matrix space $\mathbb{R}^{d\times d}$\tnoteref{label1}}

\tnotetext[label1]{Project was supported partly by National Natural Science Foundation of China (No. 11071112) and PAPD of Jiangsu Higher Education Institutions.}%

\author{Xiongping Dai}
\ead{xpdai@nju.edu.cn}
\address{Department of Mathematics, Nanjing University, Nanjing 210093, People's Republic of China}

\begin{abstract}
In this paper, we present some checkable criteria for the spectral finiteness of a finite subset of the real $d\times d$ matrix space $\mathbb{R}^{d\times d}$, where $2\le d<\infty$.
\end{abstract}

\begin{keyword}
Generalized/joint spectral radius\sep spectral finiteness.

\medskip
\MSC[2010] Primary 15B52, 15A18; Secondary 65F15\sep 93D20.
\end{keyword}

\end{frontmatter}

\section{Introduction}\label{sec1}%

Throughout this paper, by $\rho(M)$ we mean the usual spectral radius of a real square matrix $M\in\mathbb{R}^{d\times d}$, where $2\le d<+\infty$. For an arbitrary finite family of real matrices
\bean
\bA=\{A_1,\dotsc,A_K\}\subset\mathbb{R}^{d\times d},
\eean
its \textit{generalized spectral radius $\pmb{\rho}(\bA)$}, first introduced by Daubechies and Lagarias in \cite{DL92-01}, is defined by
\bean
\pmb{\rho}(\bA)=\sup_{n\ge1}\max_{M\in\bA^n}\sqrt[n]{\rho(M)}\quad\left(\,=\limsup_{n\to+\infty}\max_{M\in\bA^n}\sqrt[n]{\rho(M)}\right),
\eean
where
\begin{equation*}
\bA^n=\left\{\stackrel{n\textrm{-folds}}{\overbrace{M_1\dotsm M_n}}\,\bigg\vert\, M_i\in\bA\textrm{ for }1\le i\le n\right\}\quad \forall n\ge1.
\end{equation*}
According to the Berger-Wang spectral formula~\cite{BW92} (also see \cite{Els, Dai-JMAA} for simple proofs), this quantity is very important for many pure and applied mathematics branches like numerical computation of matrices, differential equations, coding theory, wavelets, stability analysis of random matrix, control theory, combinatorics, and so on. See, for example, \cite{Bar, Jun}.

Therefore, the following finite-step realization property for the accurate computation of $\pmb{\rho}(\bA)$ becomes very interesting and important, because it makes the stability question algorithmically decidable; see, e.g., \cite[Proposition~2.9]{Bar}.

\begin{prob}\label{prob1}
Does there exist a finite-length word which realize $\pmb{\rho}$ for $\bA$?
In other words, does there exist any $M^*\in\bA^{n^*}$ such that $\pmb{\rho}(\bA)=\sqrt[\leftroot{-2}\uproot{2}n^*]{\rho(M^*)}$, for some $n^*\ge1$?
\end{prob}

If one can find such a word $M^*$ for some $n^*\ge1$, then $\bA$ is said to possess \textit{the spectral finiteness}.
This spectral finiteness, for any bounded $\bA$, was conjectured respectively by Pyatnitski\v{i} (see, e.g.,~\cite{PR91,SWMW07}), Daubechies and Lagarias in~\cite{DL92-01}, Gurvits in~\cite{Gur95}, and by Lagarias and Wang in~\cite{LW95}. It has been disproved first by Bousch and Mairesse in \cite{BM}, and then by Blondel \textit{et al.} in \cite{BTV}, by Kozyakin in~\cite{Koz05, Koz07}, all offered the existence of counterexamples in the case where $d=2$ and $K=2$; moreover, an explicit expression for such a counterexample has been found in the recent work of Hare \textit{et al.}~\cite{HMST}.

However, an affirmative solution to Problem~\ref{prob1} is very important; this is because it implies an effective computation of $\pmb{\rho}(\bA)$ and decidability of stability by only finitely many steps of computations. There have been some sufficient (and necessary) conditions for the spectral finiteness for some types of $\bA$, based on and involving Barabanov norms, polytope norms, ergodic theory or some limit properties of $\bA$, for example, in Gurvits~\cite{Gur95}, Lagarias and Wang~\cite{LW95}, Guglielmi, Wirth and Zennaro~\cite{GWZ05}, Kozyakin~\cite{Koz07}, Dai, Huang and Xiao~\cite{DHX-pro}, and Dai and Kozyakin~\cite{DK}. But these theoretic criteria seem to be difficult to be directly employed to judge whether or not an explicit family $\bA$ or even a pair $\{A,B\}\subset\mathbb{R}^{2\times 2}$ have the spectral finiteness.

From literature, as far we know, there are only few results on such an explicit family of matrices $\bA$ as follows.

\begin{thm-A}[{Theys~\cite{Theys}, also \cite{JB08}}]
If $A_1,\dotsc,A_K\in\mathbb{R}^{d\times d}$ are all symmetric matrices, then $\bA$ has the spectral finiteness such that
$\brho(\bA)=\max_{1\le k\le K}\rho(A_k)$.
\end{thm-A}

A more general version of this theorem is the following.

\begin{thm-B}[{Barabanov~\cite[Proposition~2.2]{Bar}}]
If a finite set $\bA\subset\mathbb{R}^{d\times d}$ only contains normal matrices, then the spectral finiteness holds.
\end{thm-B}

For a matrix $A$, by $A^T$ we mean its transpose matrix. Another generalization of Theorem~A is the following.

\begin{thm-C}[{Plischke and Wirth~\cite[Proposition~18]{PW:LAA08}}]
If $\bA=\{A_1,\dotsc,A_K\}\in\mathbb{R}^{d\times d}$ is symmetric in the sense of $A_k^T\in\bA$ for all $1\le k\le K$, then $\bA$ has the spectral finiteness.
\end{thm-C}

Jungers and Blondel~\cite{JB08} proved that for a pair of $\{0,1\}$-matrices of $2\times 2$, the spectral finiteness holds.
A more general result than this is

\begin{thm-D}[{Cicone \textit{et al.}~\cite{CGSZ10}}]
If $A_1$ and $A_2$ are $2\times 2$ sign-matrices; that is, $A_1,A_2$ belong to $\{0,\pm1\}^{2\times 2}$, then the spectral finiteness holds for $\{A_1,A_2\}$.
\end{thm-D}

The followings are other different type of results.

\begin{thm-E}[{Br\"{o}ker and Zhou~\cite{BZ}}]
If $\bA=\{A, B\}\subset\mathbb{R}^{2\times 2}$ satisfies $\det(A)<0$ and $\det(B)<0$, then
$\brho(\bA)=\max\left\{\rho(A),\rho(B), \sqrt{\rho(AB)}\right\}$.
\end{thm-E}

\begin{thm-F}[{M\"{o}{\ss}ner~\cite{Mo}}]
If $\bA=\{L, R\}\subset\mathbb{R}^{2\times 2}$ satisfies
$L=\left(\begin{smallmatrix}0&1\\1&0\end{smallmatrix}\right)R\left(\begin{smallmatrix}0&1\\1&0\end{smallmatrix}\right)$,
then $\bA$ has the spectral finiteness with $\brho(\bA)=\max\left\{\rho(L), \sqrt{\rho(LR)}\right\}$.
\end{thm-F}

\begin{thm-G}[{Guglielmi \textit{et al.}~\cite[Theorem~4]{GMV}}]
Let $\bA=\{A,B\}$ satisfy
\begin{equation*}
A=\begin{pmatrix}a&b\\c&d\end{pmatrix}\quad\textrm{and}\quad B=\begin{pmatrix}a&-b\\-c&d\end{pmatrix},\quad \textrm{where }a,b,c,d\in\mathbb{R}.
\end{equation*}
Then $\bA$ has the spectral finiteness such that
\begin{equation*}
\brho(\bA)=
\begin{cases}
\rho(A)=\rho(B)& \textrm{if }bc\ge0,\\
\sqrt{\rho(AB)}& \textrm{if }bc<0.\end{cases}
\end{equation*}
\end{thm-G}

\begin{thm-H}[Dai \textit{et al.}~\cite{DHLX11}]
If one of $A, B\in\mathbb{R}^{d\times d}$ is of rank one, then there holds the spectral finiteness property for $\{A,B\}$.
\end{thm-H}

We will present a new criterion for the spectral finiteness of a finite subset of $\mathbb{R}^{d\times d}$,  see Theorem~\ref{thm1} in Section~\ref{sec2}, which generalizes Theorems~A,~C and G. From this we can obtain some checkable sufficient conditions for the spectral finiteness.

Finally in Section~\ref{sec3}, we will improve the main theorem of Kozyakin~\cite{Koz07} to get a sufficient and necessary condition for the spectral finiteness of a type of $2$-by-$2$ matrix set $\bA$; see Theorem~\ref{thm8}.

\section{Symmetric optimal words and the spectral finiteness}\label{sec2}

We let $\bA=\{A_1,\dotsc,A_K\}\subset\mathbb{R}^{d\times d}$ be an arbitrarily given set, where $2\le K<\infty$ and $2\le d<\infty$. By $\|\cdot\|$, we denote the usual euclidean norm of $\mathbb{R}^{d\times d}$. Let 
$$
\bK=\{1,2,\dotsc,K\}\quad \textrm{and}\quad
\bK^n=\stackrel{n\textit{-folds}}{\overbrace{\bK\times\dotsm\times\bK}}. 
$$
For any word $w=(k_1,\dotsc,k_n)\in\bK^n$ of length $n$, we write $\bA(w)=A_{k_1}\dotsm A_{k_n}\in\bA^n$.

A word $w^*=(k_1^*,\dotsc,k_n^*)\in\bK^n$ of length $n$ is called an \textit{$(\bA,n)$-optimal word}, provided that it satisfies condition:
\begin{equation*}
\|\bA(w^*)\|=\max_{w\in\bK^n}\|\bA(w)\|.
\end{equation*}

This section is mainly devoted to proving the following criterion for the spectral finiteness of $\bA$, which generalizes Theorems~A and C and the first part of Theorem~G.

\begin{theorem}\label{thm1}
Let $\bA=\{A_1,\dotsc,A_K\}\subset\mathbb{R}^{d\times d}$. If there exists an $(\bA,n^*)$-optimal word $w^*$ with $\bA(w^*)^T\bA(w^*)\in\bA^{2n^*}$ $(\textrm{resp. } \bA(w^*)\bA(w^*)^T\in\bA^{2n^*})$ for some $n^*\ge1$, then $\bA$ has the spectral finiteness such that
\begin{equation*}
\brho(\bA)=\sqrt[\leftroot{-3}n^*]{\|\bA(w^*))\|}=\sqrt[\leftroot{-3}2n^*]{\rho\left(\bA(w^*)^T\bA(w^*)\right)}\quad\left(=\sqrt[\leftroot{-3}2n^*]{\rho\left(\bA(w^*)\bA(w^*)^T\right)}\right).
\end{equation*}
\end{theorem}

\begin{proof}
Let $w^*$ be an $(\bA,n^*)$-optimal word of length $n^*$, which is such that $\bA(w^*)^T\bA(w^*)\in\bA^{2n^*}$ (resp. $\bA(w^*)\bA(w^*)^T\in\bA^{2n^*}$), for some $n^*\ge1$. Then from the Berger-Wang spectral formula~\cite{BW92}, it follows that
\begin{equation*}\begin{split}
\brho(\bA)&=\inf_{n\ge1}\max_{w\in\bK^n}\sqrt[n]{\|\bA(w)\|}\le\sqrt[\leftroot{-3}n^*]{\|\bA(w^*))\|}\\
&=\sqrt[\leftroot{-3}2n^*]{\rho\left(\bA(w^*)^T\bA(w^*)\right)}\\
&=\sqrt[\leftroot{-3}2n^*]{\rho\left(\bA(w^*)\bA(w^*)^T\right)}\\
&\le\brho(\bA).
\end{split}\end{equation*}
This implies the desired result and ends the proof of Theorem~\ref{thm1}.
\end{proof}

For the case where $\bA$ is symmetric as in Theorem~C, one can find an $(\bA,1)$-optimal word $w^*$ such that both $\bA(w^*)^T\bA(w^*)$ and $\bA(w^*)\bA(w^*)^T$ belong to $\bA^{2}$. On the other hand, the following simple example shows our Theorem~\ref{thm1} is an essential extension of Theorem~C.

\begin{example}\label{example2}
Let $\bA$ consist of the following three matrices:
\begin{equation*}
A_1=\begin{pmatrix}1&1&2\\0&1&1\\0&0&1\end{pmatrix},\quad A_2=\begin{pmatrix}1&0&0\\1&1&0\\2&1&1\end{pmatrix},\quad \textrm{and}\quad A_3=\begin{pmatrix}\cos\alpha&\sin\alpha&0\\-\sin\alpha&\cos\alpha&0\\0&0&\sqrt{\frac{3-\sqrt{5}}{2}}\end{pmatrix}.
\end{equation*}
It is evident that $\bA$ is not symmetric. However, $w^*=(1)$ is an $(\bA,1)$-optimal word such that
\begin{equation*}
\bA(w^*)^T\bA(w^*)=A_2A_1\in\bA^{2}
\end{equation*}
and so $\brho(\bA)=\sqrt{\rho(A_2A_1)}$.
\end{example}

As a consequence of Theorem~\ref{thm1}, we can obtain the following checkable criterion for the spectral finiteness of a kind of $\bA$.

\begin{cor}\label{cor3}
Let $\bA$ consist of the following $K+2$ matrices:
\begin{equation*}
A_0=\begin{pmatrix}a&b\\c&d\end{pmatrix},\;A_1=\begin{pmatrix}a_{11}&r_1b\\r_1c&d_{11}\end{pmatrix},\;\dotsc,\;A_K=\begin{pmatrix}a_{KK}&r_Kb\\r_Kc&d_{KK}\end{pmatrix},\; \textrm{and}\; B=\begin{pmatrix}b_{11}&r\sqrt{|b|}\\r\sqrt{|c|}&b_{22}\end{pmatrix},
\end{equation*}
where $r, r_1,\dotsc,r_K$ are all constants. If $bc\ge0$ and $\|B\|\le\max_{0\le i\le K}\rho(A_i)$,
then $\bA$ has the spectral finiteness and moreover
\begin{equation*}
\brho(\bA)=\max_{0\le k\le K}\rho(A_k).
\end{equation*}
\end{cor}

\begin{proof}
If $bc=0$ then the statement holds trivially. Next, we assume $bc>0$.
Let $k^*\in\{0,1,\dotsc,K\}$ be such that
\begin{equation*}
\rho(A_{k^*})=\max_{0\le k\le K}\rho(A_k),
\end{equation*}
and we put
\begin{equation*}
Q=\begin{pmatrix}q_1&0\\0&q_2\end{pmatrix}
\end{equation*}
which is such that 
$$q_1q_2\not=0\quad \textrm{and}\quad \frac{q_1}{q_2}=\sqrt{\frac{c}{b}}.$$ 
Then,
\begin{gather*}
QA_0Q^{-1}=\begin{pmatrix}a&\sqrt{bc}\\\sqrt{bc}&d\end{pmatrix},\\
QA_1Q^{-1}=\begin{pmatrix}a_{11}&r_1\sqrt{bc}\\r_1\sqrt{bc}&d_{11}\end{pmatrix},\\
\vdots\quad\vdots\quad \vdots\\
QA_KQ^{-1}=\begin{pmatrix}a_{KK}&r_K^{}\sqrt{bc}\\r_K^{}\sqrt{bc}&d_{KK}\end{pmatrix},\\
\intertext{and}
QBQ^{-1}=\begin{pmatrix}b_{11}&r\sqrt{|c|}\\r\sqrt{|b|}&b_{22}\end{pmatrix}=B^T.
\end{gather*}
So,
\begin{equation*}
\rho(A_i)=\|QA_iQ^{-1}\|\quad \textrm{for }0\le i\le K,\quad\textrm{and}\quad\|B^T\|\le\max_{0\le i\le K}\|QA_iQ^{-1}\|.
\end{equation*}
Thus, $w^*=(k^*)$ is a $(Q\bA Q^{-1},1)$-optimal word with $QA_{k^*}Q^{-1}\in Q\bA Q^{-1}$.

From Theorem~\ref{thm1}, this thus proves Corollary~\ref{cor3}.
\end{proof}

Corollary~\ref{cor3} generalizes the first part of Theorem~G stated in Section~\ref{sec1}. A special case of Corollary~\ref{cor3} is the following which is of independent interest.

\begin{cor}\label{cor4}
Let $\bA$ consist of
\begin{equation*}
A=\begin{pmatrix}\lambda_1&0\\0&\lambda_2\end{pmatrix}\quad \textrm{and}\quad B=\begin{pmatrix}a&b\\c&d\end{pmatrix}.
\end{equation*}
If $bc\ge0$, then $\brho(\bA)=\max\{\rho(A), \rho(B)\}$.
\end{cor}

Now we are naturally concerned with the following.

\begin{prob}\label{prob2}
What can we say for $\bA$ without the constraint condition $bc\ge0$ in Corollary~\ref{cor4}?
\end{prob}

First, a special case might be simply observed as follows.

\begin{prop}\label{prop5}
Let $A,B\in\mathbb{R}^{d\times d}$, where $2\le d<\infty$, be a pair of matrices such that
\bean
A=\begin{pmatrix}a_1&0&\dotsm&0\\0&a_2&\dotsm&0\\\vdots&\vdots&\ddots&\vdots\\0&0&\dotsm&a_d\end{pmatrix}\quad\textrm{and}\quad B=\begin{pmatrix}0&\dotsm&0&b_1\\0&\dotsm&b_2&0\\\vdots&{}&\vdots&\vdots\\b_d&\dotsm&0&0\end{pmatrix}.
\eean
Then $\bA=\{A,B\}$ is such that $\brho(\bA)=\max\{\rho(A), \rho(B)\}$.
\end{prop}

\begin{proof}
We will only prove the statement in the case of $d=3$, since the other case may be similarly proved.
By replacing $A$ and $B$ with $A/\brho$ and $B/\brho$ respectively if necessary, there is no loss of generality in assuming
$\brho(\bA)=1$. By contradiction, we assume
\begin{equation*}
\rho(A)=\max\{|a_1|,|a_2|, |a_3|\}<1
\end{equation*}
and
\begin{equation*}
\rho(B)=\max\left\{|b_2|,\sqrt{|b_1b_3|}\right\}<1.
\end{equation*}
Let $\{(m_k,n_k)\}_{k=1}^{+\infty}$ be an arbitrary sequence of positive integer pairs. We claim that
\begin{equation*}
\|A^{m_1}B^{n_1}A^{m_2}B^{n_2}\dotsm A^{m_k}B^{n_k}\|\to0
\end{equation*}
as $k\to+\infty$.

Indeed, the claim follows from the following simple computation:
\begin{gather*}
A^m=\begin{pmatrix}a_1^m&0&0\\0&a_2^m&0\\0&0&a_3^m\end{pmatrix},\quad
B^n=\begin{cases}\begin{pmatrix}(b_1b_3)^{n^\prime}&0&0\\0&b_2^{n^\prime}&0\\0&0&(b_3b_1)^{n^\prime}\end{pmatrix}& \textrm{if }n=2n^\prime,\\\begin{pmatrix}(b_1b_3)^{n^\prime}&0&0\\0&b_2^{n^\prime}&0\\0&0&(b_3b_1)^{n^\prime}\end{pmatrix}B& \textrm{if }n=2n^\prime+1;\end{cases}
\end{gather*}
and for any constants $q_i,c_i, d_i$ for $i=1,2,3$,
\begin{gather*}
\begin{pmatrix}q_1&0&0\\0&q_2&0\\0&0&q_3\end{pmatrix}\begin{pmatrix}0&0&c_1\\0&c_2&0\\c_3&0&0\end{pmatrix}
=\begin{pmatrix}0&0&q_1c_1\\0&q_2c_2&0\\q_3c_3&0&0\end{pmatrix},\\
\intertext{and}
\begin{pmatrix}0&0&c_1\\0&c_2&0\\c_3&0&0\end{pmatrix}
\begin{pmatrix}0&0&d_1\\0&d_2&0\\d_3&0&0\end{pmatrix}=\begin{pmatrix}c_1d_3&0&0\\0&c_2d_2&0\\0&0&c_3d_1\end{pmatrix}.
\end{gather*}
Then, this claim is a contradiction to $\brho(\bA)=1$ and so it implies that $\brho(\bA)=\max\{\rho(A), \rho(B)\}$.
\end{proof}

It should be noted that although $\rho(B)<1$ and $\|A\|<1$ in the proof of Proposition~\ref{prop5} under the contradiction assumption, yet $\|B\|>1$ possibly
happens; for example,
\begin{equation*}
B=\begin{pmatrix}0&0&6/5\\0&4/5&0\\2/5&0&0\end{pmatrix}
\end{equation*}
is such that $\rho(B)=4/5<1$ but $\|B\|=6/5>1$. This is just the nontrivial point of the above proof of Proposition~\ref{prop5}.

For Problem~\ref{prob2}, we cannot, however, expect a general positive solution as disproved by the following counterexample.

\begin{example}\label{example6}
Let
\begin{equation*}
A_0=\alpha\begin{pmatrix} -3&3.5\\-4&4.5
\end{pmatrix}\quad \textrm{and}\quad A_1=\beta\begin{pmatrix}0.5&0\\0&1\end{pmatrix}
\end{equation*}
where $\alpha>0, \beta>0$, and $bc=-14<0$. Then $\bA=\{A_0,A_1\}$ cannot be simultaneously symmetrized and there exists a pair of $\alpha,\beta$ so that $\bA$ has no the spectral finiteness.
\end{example}

\begin{proof}
Putting $Q=\left(\begin{smallmatrix}-0.5& 1 \\  0 & 1\end{smallmatrix}\right)$,
we have
\begin{equation*}
B_0:=Q^{-1}A_0Q=\alpha\begin{pmatrix}1 & 0 \\  2 & 0.5
\end{pmatrix}
\quad\textrm{and}\quad
B_{1}:=Q^{-1}A_1Q=\beta\begin{pmatrix}0.5 & 1 \\  0 & 1\end{pmatrix}.
\end{equation*}
According to Kozyakin~\cite[Theorem~10, Lemma~12 and Theorem~6]{Koz07}, it follows that there always exists a pair of real numbers $\alpha>0,\beta>0$ such that $\{B_0, B_1\}$ and so $\bA$ do not have the spectral finiteness.

Thus, if $\{A_0,A_1\}$ might be simultaneously symmetrized for some pair of $\alpha>0,\beta>0$, then $\{A_0, A_1\}$ and hence $\{B_0, B_1\}$ have the spectral finiteness from Theorem~A, for all $\alpha>0,\beta>0$. This is a contradiction. Therefore, $\{A_0,A_1\}$ cannot be simultaneously symmetrized for all $\alpha>0,\beta>0$.

This proves the statement of Example~\ref{example6}.
\end{proof}

Meanwhile this argument shows that the constraint condition ``$bc\ge0$" in Corollary~\ref{cor3} and even in Corollary~\ref{cor4} is crucial for the spectral finiteness in our situation.

Given an arbitrary set $\bA=\{A_1,\dotsc,A_K\}\subset\mathbb{R}^{d\times d}$, although its periodic stability implies that it is stable almost surely in terms of arbitrary Markovian measures as shown in Dai, Huang and Xiao~\cite{DHX11-aut} for the discrete-time case and in Dai~\cite{Dai-JDE} for the continuous-time case, yet its absolute stability is generally undecidable; see, e.g., Blondel and Tsitsiklis~\cite{BT97, BT00, BT00-aut}.
However, Corollary~\ref{cor3} is equivalent to the statement\,---\,``periodic stability $\Rightarrow$ absolute stability'', under suitable additional conditions.

\begin{prop}\label{prop7}
Let $\bA$ consist of the following $K+2$ matrices:
\begin{equation*}
A_0=\begin{pmatrix}a&b\\c&d\end{pmatrix},\;A_1=\begin{pmatrix}a_{11}&r_1b\\r_1c&d_{11}\end{pmatrix},\;\dotsc,\;A_K=\begin{pmatrix}a_{KK}&r_Kb\\r_Kc&d_{KK}\end{pmatrix},\; \textrm{and}\; B=\begin{pmatrix}b_{11}&r\sqrt{|b|}\\r\sqrt{|c|}&b_{22}\end{pmatrix},
\end{equation*}
where $r, r_1,\dotsc,r_K$ are all constants, such that $bc\ge0$ and $\|B\|\le\max_{0\le i\le K}\rho(A_i)$.
Then $\bA$ is absolutely stable if and only if $\rho(A_k)<1$ for all $0\le k\le K$.
\end{prop}

\begin{proof}
The statement is obvious and we omit the details here.
\end{proof}

In fact, the absolute stability of $\bA$ is decidable in the situation of Theorem~\ref{thm1}.

\section{Kozyakin's model}\label{sec3}

In \cite{Koz07}, Kozyakin systemly considered the spectral finiteness of $\bA$ which consists of the following two matrices:
\begin{equation*}
A_0=\alpha\begin{pmatrix}a&b\\0&1\end{pmatrix}\quad \textrm{and} \quad A_1=\beta\begin{pmatrix}1&0\\c&d\end{pmatrix},
\end{equation*}
where $a,b,c,d,\alpha$, and $\beta$ are all real constants, such that
\begin{equation*}
\alpha,\beta>0\quad \textrm{and}\quad bc\ge1\ge a>0,\;d>0.\leqno{(\mathrm{K})}
\end{equation*}
Let $\brho=\brho(\bA)$. We first note that from \cite{Bar} there exists a Barabanov norm $\pmb{\|}\cdot\pmb{\|}$ on $\mathbb{R}^2$; i.e.,
$$
\brho\pmb{\|}x\pmb{\|}=\max\{\pmb{\|}A_0x\pmb{\|}, \pmb{\|}A_1x\pmb{\|}\}\quad\forall x\in\mathbb{R}^2.
$$
And so for any $x_0\in\mathbb{R}^2\setminus\{0\}$, one can find a corresponding (B-extremal) switching law
$$
\mathfrak{i}_{\bcdot}(x_0)\colon\{1,2,\dotsc\}\rightarrow\{0,1\}
$$
such that
$$\pmb{\|}A_{\mathfrak{i}_n(x_0)}\dotsm A_{\mathfrak{i}_1(x_0)}x_0\pmb{\|}=\pmb{\|}x_0\pmb{\|}\brho^n\quad\forall n\ge1.$$
Then from Kozyakin~\cite[Theorem~6]{Koz07}, it follows that there exists the limit
\begin{equation*}
\sigma(\bA):=\lim_{n\to\infty}\frac{1}{n}\sum_{k=1}^n\mathfrak{i}_k(x_0),
\end{equation*}
called the \textit{switching frequency} of $\bA$, which does not depend on the choices of $x_0$ and the (B-extremal) switching law $\mathfrak{i}_{\bcdot}(x_0)$.

Kozyakin (cf.~\cite[Theorem~10]{Koz07}) asserted that if $\sigma(\bA)$ is irrational, then $\bA$ does not have the spectral finiteness. We now show that this is also necessary.

\begin{theorem}\label{thm8}
Under condition $(\mathrm{K})$, $\bA$ has the spectral finiteness iff its switching frequency $\sigma(\bA)$ is rational.
\end{theorem}

\begin{proof}
If $\sigma(\bA)$ is an irrational number, then \cite[Theorem~10]{Koz07} follows that $\bA$ does not have the spectral finiteness.
Next, assume $\sigma(\bA)$ is rational. Then \cite[Theorem~6]{Koz07} implies that one can find some $x_0\in\mathbb{R}^2\setminus\{0\}$ and a corresponding
periodic switching law, say
\begin{equation*}
\mathfrak{i}_{\bcdot}(x_0)=(\uwave{i_1i_2\dotsm i_\pi}\,\uwave{i_1i_2\dotsm i_\pi}\,\uwave{i_1i_2\dotsm i_\pi}\,\dotsm),
\end{equation*}
such that
$$\pmb{\|}A_{\mathfrak{i}_n(x_0)}\dotsm A_{\mathfrak{i}_1(x_0)}x_0\pmb{\|}=\pmb{\|}x_0\pmb{\|}\brho^n\quad\forall n\ge1,$$
where $\brho=\brho(\bA)$. Therefore, it holds that
$$\pmb{\|}(A_{i_\pi}\dotsm A_{i_1})^n\pmb{\|}\ge\brho^{n\pi}\quad\forall n\ge1.$$
Moreover, from the classical Gel'fand spectral formula we have
\bean
\rho(A_{i_\pi}\dotsm A_{i_1})=\lim_{n\to\infty}\sqrt[n]{\pmb{\|}(A_{i_\pi}\dotsm A_{i_1})^n\pmb{\|}}\ge\brho^{\pi}.
\eean
Thus, $\brho(\bA)=\sqrt[\pi]{\rho(A_{i_\pi}\dotsm A_{i_1})}$, which means the spectral finiteness.

This completes the proof of Theorem~\ref{thm8}.
\end{proof}

This result improves \cite[Theorem~10]{Koz07} and it should be convenient for applications. Let us consider an explicit example.

\begin{example}\label{example9}
Let $\bB=\{B_0,B_1\}$ be such that
\bean
B_0=\begin{pmatrix}a&b\\0&1\end{pmatrix}\quad \textrm{and} \quad B_1=\begin{pmatrix}1&0\\c&d\end{pmatrix},
\eean
where $a,b,c,d\in\mathbb{R}$.
\end{example}

We will divide our arguments into several cases.

1). If $ad=0$ then we have either $\mathrm{rank}(B_0)=1$ or $\mathrm{rank}(B_1)=1$ and so $\bB$ has the spectral finiteness from Theorem~H stated in Section~\ref{sec1}.

2). If $bc=0$ then $\bB$ has the spectral finiteness from Corollary~\ref{cor4} stated in Section~\ref{sec2}.

3). If $a<0$ and $d<0$, then $\bB$ has the spectral finiteness from Theorem~E stated in Section~\ref{sec1}.

4). If $a=d$ and $b=c$, then $\bB$ has the spectral finiteness from Theorem~F stated in Section~\ref{sec1} such that $\brho(B)=\max\left\{\rho(B_0), \sqrt{\rho(B_0B_1)}\right\}$.

5). Next, let $ad\not=0, bc\not=0$, and define
\begin{equation*}
Q=\begin{pmatrix}\frac{a-1}{b}&1\\0&1\end{pmatrix}.
\end{equation*}
When $a\not=1$, we can obtain that
\begin{equation*}
QB_0Q^{-1}=\begin{pmatrix}a&0\\0&1\end{pmatrix}\quad \textrm{and}\quad QB_1Q^{-1}=\begin{pmatrix}1+\frac{bc}{a-1}&\frac{(d-1)(a-1)-bc}{a-1}\\\frac{bc}{a-1}&d-\frac{bc}{a-1}\end{pmatrix}.
\end{equation*}
Note that
\begin{equation*}
\frac{(d-1)(a-1)-bc}{a-1}\times\frac{bc}{a-1}\ge0\quad
\mathrm{iff}\quad
[(1-a)(1-d)-bc]\times bc\ge0.
\end{equation*}
Hence, if
\begin{equation*}
[(1-a)(1-d)-bc]\times bc\ge0,
\end{equation*}
then from Corollary~\ref{cor4}, it follows that $\bB$ has the spectral finiteness such that 
$$\brho(B)=\max\{\rho(B_0), \rho(B_1)\}.$$

6). If $a=d=1$ and $bc\ge1$, then $\bB$ has the spectral finiteness from Theorem~\ref{thm8}.
Indeed in this case, \cite[Lemma~12]{Koz07} implies that the switching frequency $\sigma(\bB)=\frac{1}{2}$ is rational, and then Theorem~\ref{thm8}
implies the spectral finiteness of $\bB$.

We notice that our cases 1)\,--\,5) are beyond Kozyakin's condition $(\mathrm{K})$.
\section*{\textbf{Acknowledgments}}%
\noindent The author would like to thank professors Y.~Huang and M.~Xiao for some helpful discussion, and is particularly grateful to professor Victor Kozyakin for his helpful comments to Theorem~\ref{thm8}.
\bibliographystyle{amsplain}

\end{document}